\documentclass[10pt]{article}
\usepackage{amsmath, amssymb, amsfonts, amsthm, mathdots, mathtools, xargs}

\usepackage[colorlinks]{hyperref}
\usepackage{amsrefs}

\title{On the length of non-solutions 
to equations with constants 
in some linear groups}
\author{Henry Bradford, Jakob Schneider 
and Andreas Thom}

\newtheorem{thm}{Theorem}[section]
\newtheorem{lem}[thm]{Lemma}
\newtheorem{propn}[thm]{Proposition}
\newtheorem{coroll}[thm]{Corollary}
\newtheorem{defn}[thm]{Definition}
\newtheorem{ex}[thm]{Example}

\newtheorem{rmrk}[thm]{Remark}
\newtheorem{qu}[thm]{Question}

\DeclareMathOperator{\Ad}{Ad}

\DeclareMathOperator{\diam}{diam}

\DeclareMathOperator{\Homeo}{Homeo}

\DeclareMathOperator{\Orb}{Orb}

\DeclareMathOperator{\PGL}{PGL}
\DeclareMathOperator{\PSL}{PSL}

\DeclareMathOperator{\SL}{SL}

\DeclareMathOperator{\Sym}{Sym}
\DeclareMathOperator{\tr}{tr}

\begin{document}

\maketitle

\begin{abstract}
We show that for any finite-rank free group $\Gamma$, 
any word-equation in one variable of length $n$ 
with constants in $\Gamma$ 
fails to be satisfied by some element of $\Gamma$ 
of word-length $O(\log (n))$. 
By a result of the first author, 
this logarithmic bound cannot be improved upon 
for any finitely generated group $\Gamma$. 
Beyond free groups, our method 
(and the logarithmic bound) 
applies to a class of groups 
including $\PSL_d(\mathbb{Z})$ for all $d \geq 2$, 
and the fundamental groups of all 
closed hyperbolic surfaces and $3$-manifolds. 
Finally, using a construction of Nekrashevych, 
we exhibit a finitely generated group $\Gamma$ and 
word-equations with constants in $\Gamma$ for which 
every non-solution in $\Gamma$ is of word-length 
strictly greater than logarithmic. 

%In ???, the first author showed that 
%in any finitely generated group $\Gamma$, 
%there exist $w_n \in \Gamma \ast \mathbb{Z}$ 
%of word-length $O(n)$, 
%such that the associated word-map 
%$w_n : \Gamma \rightarrow \Gamma$ 
%vanishes identically on 
%the ball of radius $\log (n)$ 
%in the word-metric on $\Gamma$. 
%In this note we show that the logarithmic bound 
%is sharp for nonabelian free groups; 
%more generally for all nonelementary Kleinian groups, 
%and for $\PSL_d (\mathbb{Z})$ ($d \geq 2$). 
\end{abstract}

\section{Introduction}

For $\Gamma$ any group and any $g \in \Gamma$, 
there is a unique homomorphism 
$\pi_g : \Gamma \ast \mathbb{Z} \rightarrow \Gamma$ 
restricting to the identity on $\Gamma$ 
and sending a specified generator $x$ of $\mathbb{Z}$ 
to $g$. 
We call $w$ a \emph{mixed identity} 
(or an \emph{identity with constants}) for $\Gamma$ 
if $w$ is nontrivial but $\pi_g (w)=e_{\Gamma}$ 
for all $g \in \Gamma$. 
Equivalently, we may define a \emph{word-map} 
$w : \Gamma \rightarrow \Gamma$ 
by $w(g) = \pi_g (w)$; 
$w$ is then a mixed identity for $\Gamma$ 
iff $w^{-1} (e_{\Gamma}) = \Gamma$. 

We describe $\Gamma$ as \emph{MIF} 
(mixed identity-free) if there are no mixed 
identities for $\Gamma$ in $\Gamma \ast \mathbb{Z}$. 
Nonabelian free groups are MIF 
(a fact due to Baumslag \cite{Baum}), 
as more generally are all torsion-free 
nonelementary Gromov hyperbolic groups 
(see \cite{AmRe} and the references therein). 
Being MIF imposes significant structural 
restrictions on $\Gamma$: 
an MIF group has no nontrivial 
finite conjugacy classes, 
and admits no nontrivial decomposition 
as a direct product. 
Moreover, MIF is inherited by nontrivial normal 
subgroups. 

Given a finite generating set $S$ for $\Gamma$, 
if $\Gamma$ is MIF then for any 
$w \in \Gamma \ast \mathbb{Z}$ 
there exists $g \in \Gamma$ 
of minimal word-length $\lvert g \rvert_S$ 
satisfying $w(g)\neq e$. 
Intuitively, 
the greater this minimal word-length should be, 
the harder it is to verify that 
$w$ is not a mixed identity for $\Gamma$. 
The goal of this note is to study 
how $\lvert g \rvert_S$ can depend 
on the length of $w$. 
Following \cite{Brad}, 
we formalize this as follows. 
The \emph{complexity} of $w \in\Gamma\ast\mathbb{Z}$ 
(with respect to $S$) is given by: 
\begin{equation*}
\chi_{\Gamma} ^S (w) = \min \lbrace \lvert g \rvert_S 
: g \in \Gamma, w(g)\neq e_{\Gamma} \rbrace
\end{equation*}
(with the convention that 
$\chi_{\Gamma} ^S (w) = +\infty$ 
if $w$ is a mixed identity for $\Gamma$). 
Here $\lvert g \rvert_S$ denotes the 
\emph{word-length} of $g$ with respect to the 
generating set $S$. 
The \emph{MIF growth function} 
$\mathcal{M}_{\Gamma} ^S 
: \mathbb{N} \rightarrow 
\mathbb{N} \cup \lbrace +\infty \rbrace$ 
of $\Gamma$ (with respect to $S$) 
is then given by: 
\begin{equation*}
\mathcal{M}_{\Gamma} ^S (n) = 
\max \big\lbrace \chi_{\Gamma} ^S (w) 
: w \in \Gamma \ast \mathbb{Z} , \lvert w \rvert_{S\cup\lbrace x \rbrace} \leq n \big\rbrace. 
\end{equation*}
Intuitively, a group of slow MIF growth is 
``strongly MIF'': it is easy to verify that 
a given $w$ is not a mixed identity for $\Gamma$, 
because there exists a witness of short word-length. 
As the dependence of $\mathcal{M}_{\Gamma} ^S$ 
on $S$ is slight, 
we suppress it from our 
notation for the remainder of the Introduction 
(with the understanding that all implied 
constants may depend on $S$). 

Since MIF places such powerful structural 
restrictions on a group, 
it is no surprise that groups of slow MIF growth 
are difficult to identify, and indeed, 
there are no groups of bounded MIF growth. 

\begin{thm}[\cite{Brad} Theorem 1.9] \label{GeneralLowerBoundThm}
For any finitely generated group $\Gamma$, 
$\mathcal{M}_{\Gamma}(n) \gg \log (n)$. 
\end{thm}

In light of Theorem \ref{GeneralLowerBoundThm}, 
we shall refer to a finitely generated group 
$\Gamma$ as \emph{sharply MIF} if its MIF 
growth function $\mathcal{M}_{\Gamma}$ 
satisfies $\mathcal{M}_{\Gamma}(n) \ll \log (n)$. 
Prior to this note, 
we did not have any examples of sharply MIF groups. 
In fact the only previously existing explicit upper 
bound on MIF growth was: 
if $\Gamma$ is a finite-rank nonabelian free group 
then $\mathcal{M}_{\Gamma}(n) \ll n \log (n)$ 
(\cite{Brad} Theorem 1.10). 
Employing different methods from those of \cite{Brad}, 
we improve this result for free groups, 
and extend it to many other groups besides. 

\begin{thm}[Theorem \ref{KleinMainThm}] \label{IntroKleinMainThm}
Let $\Gamma$ be a finitely generated 
torsion-free nonelementary Kleinian group. 
Then $\Gamma$ is sharply MIF. 
\end{thm}

Since free groups are Kleinian, 
we can consequently improve upon 
the result from \cite{Brad}. 

\begin{coroll}[Theorem \ref{FreeGroupsThm}] \label{IntroFreeGroupsThm}
Let $\Gamma$ be a free group of finite rank 
at least $2$. 
Then $\Gamma$ is sharply MIF. 
\end{coroll}

More generally, Theorem 
\ref{IntroKleinMainThm} implies the following. 

\begin{coroll} \label{IntroMfldsCoroll}
Let $M$ be either 
(i) a closed orientable surface of genus at least $2$, 
and (ii) a closed orientable 
hyperbolic $3$-manifold. 
Then $\Gamma$ is sharply MIF. 
\end{coroll}

As these two corollaries illustrate, 
there are large families of torsion-free 
Gromov hyperbolic groups with logarithmic MIF growth. 
One may ask whether these are instances 
of a more general phenomenon. 

\begin{qu} \label{IntroHypGroupsQ}
Let $\Gamma$ be a torsion-free nonelementary 
Gromov hyperbolic group. 
Must $\Gamma$ be sharply MIF? 
\end{qu}

Of necessity, any proof of an affirmative 
answer to Question \ref{IntroHypGroupsQ} 
would require radically different methods from ours, 
since we make essential use of the abundant 
supply of finite quotients of Kleinian groups. 
As such, there is no obvious way to 
apply our proof-strategy to 
groups which are not known to be residually finite. 

More generally still, 
if $\mathbb{K}$ is an algebraically closed field, 
and $\mathbb{G}$ is an algebraic group 
defined over $\mathbb{K}$, 
then for any $w \in \mathbb{G} \ast \mathbb{Z}$, 
the associated word-map 
$w : \mathbb{G} \rightarrow \mathbb{G}$ 
is a morphism of algebraic varieties, 
so that $w^{-1} (e_{\mathbb{G}}) \subseteq \mathbb{G}$ 
is Zariski-closed. 
Therefore, if $\mathbb{G}$ is MIF 
then so is any Zariski-dense subgroup 
$\Gamma \leq \mathbb{G}$. 
Here we focus on the case 
$\mathbb{G}=\PGL_d(\mathbb{C})$: 
by Theorem 5 of \cite{Toma}, 
$\PGL_d(\mathbb{C})$ is MIF. 
Our most general result makes gives 
sharp bounds on the MIF growth 
of certain finitely generated Zariski-dense 
subgroups of $\PGL_d(\mathbb{C})$ 
(see Theorem \ref{MainNoFieldThm} below). 

\begin{thm} \label{IntroMainThm}
Let $d \geq 2$ and let $K$ be a number field. 
Let $\Gamma \leq \PSL_d(K)$ be finitely generated 
and Zariski-dense in $\PGL_d(\mathbb{C})$. 
If $d \geq 3$ suppose that: 
\begin{equation} \label{TraceEq}
K = \mathbb{Q}\big( \lbrace \tr(\Ad (g)) : g \in \Gamma \rbrace \big). 
\end{equation} 
Then $\Gamma$ is sharply MIF. 
\end{thm}

Here $\Ad$ denotes the adjoint representation of 
$\mathbb{G}$ on the associated Lie algebra. 
It is well known that every nonelementary 
Kleinian group satisfies the conclusion 
of Theorem \ref{IntroMainThm} with $d=2$, 
so Theorem \ref{IntroKleinMainThm} follows 
immediately. 
Other examples 
which clearly satisfy the conditions of 
Theorem \ref{IntroMainThm} are the following 
(see Section \ref{AppSect} below). 

\begin{ex} \label{IntroPLSdEx}
For every $d \geq 2$, 
$\PSL_d (\mathbb{Z})$ is sharply MIF; 
so too is every finitely generated 
subgroup $\Gamma$ of $\PSL_d (\mathbb{Z})$ 
which is Zariski-dense in $\PGL_d(\mathbb{C})$. 
\end{ex}

There has been much interest in recent years 
in so-called \emph{thin} groups, 
that is, 
subgroups of an arithmetic group 
which have infinite index but are nevertheless 
Zariski-dense in the ambient algebraic group. 
Specifically, there are many interesting open 
questions asking which group-theoretic 
properties a thin group must satisfy \cite{BreOh}. 
Theorem \ref{IntroMainThm} shows that, 
beside $\PSL_d (\mathbb{Z})$ itself, 
any finitely generated thin subgroup 
of $\PSL_d (\mathbb{Z})$ must be sharply MIF. 

\begin{rmrk}
\normalfont
One could also define MIF growth in terms of 
words with constants in $k$ variables, 
for any $k \in \mathbb{N}$, 
by considering the complexity in $\Gamma$ 
of elements of $\Gamma \ast \mathbb{F}_k$, 
where $\mathbb{F}_k$ is a free group of rank $k$ 
(with the complexity of $w\in\Gamma\ast\mathbb{F}_k$ 
being defined in terms of an appropriate metric 
on $\Gamma^k$). 
However, whether or not the the group is MIF, 
and the asymptotic behaviour 
of the MIF growth function 
will be the same, irrespective of which (fixed) 
value of $k$ we use to define it 
(see Lemma 2.2 of \cite{BST}); 
for instance, whether or not $\Gamma$ is sharply MIF 
does not depend on the value of $k$. 
For simplicity we therefore 
restrict to the one-variable case. 
\end{rmrk}

All of our upper bounds on MIF growth are based on 
the existence of a rich supply of finite quotients 
of $\Gamma$ which have no short mixed identities, 
and with good expansion properties. 
The lengths of mixed identities for finite groups 
were studied by the authors in \cite{BST}, 
where the following result was proved 
(see Theorem 3 of \cite{BST}, 
of which the following is a special case). 

\begin{thm} \label{PSLThmBST}
There exists $c > 0$ such that for all 
primes $p$ and all $d \geq 2$, 
$\PSL_d (p)$ has no mixed identity 
of length $\leq cp$. 
\end{thm}

The basic proof-strategy of all our upper bounds 
on $\mathcal{M}_{\Gamma}$ in this note is the same. 
To ease the understanding of the 
general proof to follow, let us sketch 
the argument in the special case 
of $\Gamma$ a nonabelian free group, 
with finite free basis $S$. 
First, we faithfully represent $\Gamma$ 
as a subgroup of $\SL_2 (\mathbb{Z})$. 
Let $w(x) = x^{a_0} c_1 x^{a_1} \cdots c_k x^{a_k}$ 
be a nontrivial reduced word in $\Gamma\ast\mathbb{Z}$ 
so that all $c_i \in \SL_2 (\mathbb{Z})$ 
are nontrivial (indeed non-central). 
For $p$ a prime number, let 
$\pi_p : \SL_2 (\mathbb{Z}) \rightarrow \SL_2(p)$ 
be reduction modulo $p$. 
By the Prime Number Theorem, 
there exists $p$ of moderate size such that 
all $\pi_p (c_i)\in\SL_2 (p)$ are still non-central. 
By Theorem \ref{PSLThmBST}, 
there exists $\overline{g} \in \SL_2(p)$ 
for which $\overline{w} (\overline{g})$ 
is non-central in $\SL_2(p)$, 
where $\overline{w}(x) 
= x^{a_0} \pi_p(c_1) x^{a_1} \cdots \pi_p(c_k) x^{a_k} 
\in \SL_2(p) \ast \mathbb{Z}$. 
We would like to lift $\overline{g}$ 
to $g \in \Gamma$, so that $w(g) \neq e$. 
By deep results on the expansion 
properties of $\SL_2(p)$, 
not only is the restriction of $\pi_p$ to $\Gamma$ 
surjective (so that such a lift $g$ exists), 
but the Cayley graph of $\SL_2 (p)$ 
with respect to generating set $\pi_p(S)$ 
has small diameter, 
so that we can choose that lift $g$ to have 
small word-length with respect to $S$. 

In the setting of Theorem \ref{IntroMainThm}, 
we will be working over a number field $K$, 
so instead of reductions modulo 
a rational prime, 
we shall be considering reductions 
$\pi_{\mathcal{P}}$ modulo a prime element 
$\mathcal{P}$ in the ring of integers of $K$. 
If $d \geq 3$, the surjectivity of 
$\pi_{\mathcal{P}}$ is not such a straightforward matter as it is for $d=2$; 
the technical hypothesis (\ref{TraceEq}) 
is assumed in Theorem \ref{IntroMainThm}, 
so that surjectivity is guaranteed, 
by an appropriate ``Strong Approximation'' Theorem. 

In light of our results as described above, 
one might wonder whether in fact \emph{every} 
finitely generated MIF group is sharply MIF. 
It transpires that this is not the case. 
It was noted already in \cite{Brad} (Remark 9.3) 
that to find a counterexample, 
it would be sufficient to find a finitely 
generated MIF group of subexponential word growth. 
Following circulation of an earlier version 
of our results, N. Matte Bon kindly pointed out to us 
that a construction of Nekrashevych, 
based on ideas from topological dynamics, 
provides such a group. 
As such we have the following. 

\begin{thm} \label{NonSharpMIFThm}
There exists a finitely generated MIF group 
which is not sharply MIF. 
\end{thm}

%We produce a family of $\PSL_d (p)$-quotients 
%of $\Gamma$, which is sufficiently rich 
%as to effectively separate the constants 
%in a short word $w \in \Gamma \ast \mathbb{Z}$.  
%Thus, there exists a prime $p$ of moderate size, 
%such that $w$ induces a nontrivial 
%element $\overline{w} \in \PSL_d (p)\ast\mathbb{Z}$. 
%Applying Theorem \ref{PSLThmBST}, 
%we obtain $g \in \PSL_d (p)$ such that $w(g) \neq e$. 
%Finally, if the Cayley graphs of these 
%$\PSL_d (p)$-quotients have 
%small diameter (with respect to the image of a finite 
%generating set for $\Gamma$) 
%then we may lift to a short non-solution to 
%$w(x)$ in $\Gamma$. 
%For this final step, we avail ourselves 
%of the extensive literature on 
%``superstrong approximation'' 
%for linear groups in characteristic zero. 

\section{Linear groups over number fields}

Let $K$ be a number field, 
with $\lvert K : \mathbb{Q} \rvert = D$. 
Let $C_0 > 0$ (to be chosen). 
Let $P_n$ be the set of rational primes 
$p$ satisfying $C_0 n < p \leq C_0 n^2$, 
and let $Q_n \subseteq P_n$ be the subset of 
those primes which split completely over $K$. 

\begin{lem} \label{ManySplitLem}
For all $c > 0$, 
if $C_0$ is chosen sufficiently large 
(depending on $c$ and $D$) then
for all $n \gg_D 1$, 
\begin{equation*}
\prod_{p \in Q_n} p \geq \exp(c n^2)
\end{equation*}
\end{lem}

\begin{proof}
By the Prime Number Theorem, 
the product $p(x)$ of all rational primes 
at most $x$ satisfies: 
\begin{equation*}
\lim_{x \rightarrow \infty} \log (p(x))/x = 1.
\end{equation*}
By Chebotarev's Density Theorem 
the proportion of rational primes at most $C n^2$ 
which split completely over $K$ 
is at least $1/2D$ (for all $n$ larger than 
a constant depending only on $K$), 
hence the product of all such primes is 
at least $p(C_0 n^2/2D)$. 
Thus: 
\begin{equation*}
\prod_{p \in Q_n} p \geq p(C_0 n^2/2D)/p(C_0 n)
\end{equation*}
and the conclusion follows. 
\end{proof}

Let $\mathcal{O}_K$ be the ring of integers of $K$. 
Let $\sigma_1 , \ldots , \sigma_D :
K \hookrightarrow \mathbb{C}$ be the distinct 
field embeddings of $K$ into $\mathbb{C}$. 
Recall that the \emph{norm} of an element 
$x \in K$ is given by: 
\begin{equation*}
N_{K/\mathbb{Q}} (x) = \prod_{i=1} ^D \sigma_i (x). 
\end{equation*}
For $\alpha \in \mathcal{O}_K$ define: 
\begin{equation*}
\mu (\alpha) = 
\max_{1 \leq i \leq D} \lvert \sigma_i (\alpha) \rvert
\end{equation*}
Then for all $\alpha , \beta \in \mathcal{O}_K$: 
\begin{itemize}
\item[(i)] $\mu(\alpha+\beta) \leq \mu(\alpha) + \mu(\beta)$

\item[(ii)] $\mu(\alpha\beta) \leq \mu(\alpha)\mu(\beta)$

\item[(iii)] If $\alpha \neq 0$ 
then $\mu (\alpha) \geq 1$; 

\item[(iv)] $\mu (\alpha)^D 
\geq N_{K/\mathbb{Q}} (\alpha)$. 
\end{itemize}
Thus for $p$ a rational prime 
and $\mathcal{P} \in \mathcal{O}_K$ 
a $K$-prime dividing $p$, 
$p$ divides $N_{K/\mathbb{Q}}(\mathcal{P})$, 
hence $\mu(\mathcal{P}) \geq p^{1/D}$. 

Let $\Gamma \leq \SL_d (K)$ be generated by 
the finite set $S$, 
which we may assume to be symmetric 
and to contain $I_d$. 
Let $\mathbb{Z} = \langle x \rangle$, 
so that $S \cup \lbrace x \rbrace$ 
generates $\Gamma \ast \mathbb{Z}$. 
Our main technical result is as follows. 

\begin{thm} \label{MainNoFieldThm}
Suppose $\Gamma \leq \SL_d (K)$ is Zariski-dense 
in $\SL_d (\mathbb{C})$. 
If $d \geq 3$, suppose that 
$K$ and $\Gamma$ satisfy condition (\ref{TraceEq}). 
There exists $C=C(S)>0$ such that the following holds. 
Let: 
\begin{equation} \label{worddefneqn}
w(x) = x^{a_0} c_1 x^{a_1} \cdots c_k x^{a_k} 
\in \Gamma \ast \mathbb{Z}
\end{equation}
for some $c_i \in \Gamma \setminus Z(\Gamma)$ 
and integers $a_i$ 
with $a_i \neq 0$ for $1 \leq i \leq k-1$. 
Suppose that $\lvert w\rvert_{S\cup\lbrace x \rbrace}
= n$. 
Then there exists $g \in \Gamma$ 
such that $\lvert g \rvert_S \leq C\log(n)$ 
and $w(g) \notin Z(\Gamma)$. 
\end{thm}

The rest of this Section is devoted to the proof 
of Theorem \ref{MainNoFieldThm}, 
and the deduction of 
Theorem \ref{IntroMainThm} from it. 
In what follows, we continue to assume 
that the hypotheses of Theorem \ref{MainNoFieldThm} 
hold. 
Note that since 
$\lvert w\rvert_{S\cup\lbrace x \rbrace}=n$, 
we have 
$\lvert c_i \rvert_S \leq n$ for $1 \leq i \leq k$. 
We may take $n$ to be larger than a quantity 
depending only on $K$ and $d$, 
which will be chosen in the course of the proof. 

For each $s \in S$, 
we may choose 
$s^{\ast} \in \mathbb{M}_d (\mathcal{O}_K)$ 
and $\delta(s) \in \mathcal{O}_K$ such that 
$s = (1/\delta(s)) s^{\ast}$ in $\mathbb{M}_d (K)$. 
We can then extend to the whole of $\Gamma$, 
by choosing, for each $g \in \Gamma$, 
a representative $g = s_1 \cdots s_{\lvert g \rvert_S}$ 
of $g$ as a word in $S$ of minimal length, 
and setting:
\begin{equation}
g^{\ast} 
= s_1 ^{\ast} \cdots s_{\lvert g \rvert_S} ^{\ast}
\text{ and }\delta(g)=
\delta(s_1) \cdots \delta(s_{\lvert g \rvert_S})
\end{equation}
so that for all $g \in \Gamma$, 
$g = (1/\delta(g))g^{\ast}$.  
Let $M$ be the maximal value of $\mu$ over 
the nonzero elements of 
$\lbrace s^{\ast} _{i,j} : s \in S , 1 \leq i,j \leq d \rbrace \cup \lbrace \delta(s) : s \in S \rbrace$. 

\begin{lem} \label{HeightBoundLem}
For any $g \in \Gamma$, 
every entry $g_{i,j} ^{\ast}$ of $g^{\ast}$ satisfies 
$\mu (g_{i,j} ^{\ast}) \leq (dM)^{\lvert g \rvert_S}$, 
and $\mu(\delta(g)) \leq M^{\lvert g \rvert_S}$. 
\end{lem}

\begin{proof}
We show that $\mu$ takes value at most $(dM)^r$ 
on all entries of products of the form 
$X = s_1 ^{\ast} \cdots s_r ^{\ast}$. 
Writing $Y = s_2 ^{\ast} \cdots s_r ^{\ast}$, 
so that $X = s_1 ^{\ast} Y$, we have: 
\begin{align*}
\mu (X_{i,j}) 
= \sum_{l=1} ^d \mu(s_{i,k} ^{\ast}) \mu(Y_{k,j}) 
\leq M \sum_{l=1} ^d \mu(Y_{k,j}) 
\leq M \sum_{l=1} ^d (dM)^{r-1} 
= (dM)^r
\end{align*}
(by induction on $r$). 
The second claim holds by 
submultiplicativity of $\mu$. 
\end{proof}

If $p$ splits completely over $K$, 
then for any $K$-prime $\mathcal{P} \in \mathcal{O}_K$ 
dividing $p$, reduction modulo $\mathcal{P}$ 
induces a surjective ring homomorphism 
$\pi_{\mathcal{P}} : 
\mathcal{O}_K \rightarrow \mathbb{F}_p$. 
If moreover $\mathcal{P}$ does not divide 
any element of $\lbrace \delta(s) : s\in S \rbrace$, 
then every nonzero entry of every element of $\Gamma$ 
is invertible modulo $\mathcal{P}$, 
and $\pi_{\mathcal{P}}$ induces a group homomorphism 
$\Gamma \rightarrow \SL_d (p)$, 
which we also denote by $\pi_{\mathcal{P}}$. 

For $G$ a group, $T \subseteq G$ 
a generating set and $r \geq 0$ let: 
\begin{equation*}
B_T(r) = \lbrace g \in G : 
\lvert g \rvert_T \leq r \rbrace
\end{equation*}
denote the closed ball of radius $r$ about the identity
in the word-metric induced by $T$ on $G$. 

\begin{lem} \label{InjRadLem}
There exists $\tilde{c}=\tilde{c}(S) > 0$ 
such that the following holds. 
Let $p \in Q_n$ and let $\mathcal{P}\in\mathcal{O}_K$ 
be a $K$-prime dividing $p$. 
Suppose $\mathcal{P}$ does not divide 
any element of $\lbrace \delta(s) : s\in S \rbrace$. 
Then the restriction of $\pi_{\mathcal{P}}$ 
to $B_S (c\log(p))$ is injective. 
\end{lem}

\begin{proof}
Note that if $h,k \in \Gamma$ satisfy 
$\pi_{\mathcal{P}} (h) = \pi_{\mathcal{P}} (k)$ 
then $g = hk^{-1} \in \ker(\pi_{\mathcal{P}})$ 
and $\lvert g \rvert_S 
\leq \lvert h \rvert_S+ \lvert k \rvert_S$. 
Hence suppose that 
$I_d \neq g \in \ker(\pi_{\mathcal{P}})$. 
Then $0 \neq X = g^{\ast} - \delta(g) I_d 
\in \mathbb{M}_d (\mathcal{O}_K)$ 
has all entries divisible by $\mathcal{P}$. 
Letting $x \in \mathcal{O}_K$ be a nonzero 
entry of $X$, we therefore have 
$\mu(x) \geq p^{1/D}$. 
On the other hand, by Lemma \ref{HeightBoundLem}, 
$\mu(x) \leq (dM)^{\lvert g \rvert_S} + M^{\lvert g \rvert_S} \leq ((d+1)M)^{\lvert g \rvert_S}$, 
and the conclusion follows. 
\end{proof}

\begin{propn} \label{GoodPrimesPropn}
Let $Q_n$ be as in Lemma \ref{ManySplitLem}. 
There exists $p \in Q_n$ and a $K$-prime 
$\mathcal{P} \in \mathcal{O}_K$ dividing $p$, 
such that: 
\begin{itemize}
\item[(i)] $\mathcal{P}$ does not divide 
any element of $\lbrace \delta(s) : s\in S \rbrace$; 

\item[(ii)] $\pi_{\mathcal{P}}:
\Gamma \rightarrow \SL_d (p)$ is surjective; 

\item[(iii)] For all $1 \leq i \leq k$, 
$\pi_{\mathcal{P}} (c_i) \in \SL_d (p)$ 
is non-central. 

\end{itemize}
\end{propn}

For the proof we require the following 
well-known auxiliary fact. 

\begin{lem} \label{SL2SubgrpLem}
For every prime $p$ and every proper subgroup 
$H < \SL_2 (p)$, either $H$ is metabelian 
or $\lvert H \rvert \mid 240$. 
\end{lem}

\begin{proof}[Proof of 
Proposition \ref{GoodPrimesPropn}]
Since $S$ is finite, there are 
only finitely many $K$-primes 
$\mathcal{P} \in \mathcal{O}_K$ for which (i) fails. 
We claim that there are only finitely many 
$\mathcal{P}$ satisfying (i) but not (ii). 
For $d \geq 3$, Theorem 10.5 of \cite{Weis} applies 
(by (\ref{TraceEq})). 
The case $d=2$ is described in 
\cite{LoLuRe} Section 3; 
we outline the argument here for the reader's 
convenience. 
By the Tits alternative, 
there exist $a , b \in \Gamma$ 
freely generating a nonabelian free subgroup 
of $\Gamma$. 
Set $g = [[a,b],[b^{-1},a]]^{240} 
\in \Gamma$ (so that $g$ is nontrivial). 
If $\pi_{\mathcal{P}}$ is not 
surjective, then by Lemma \ref{SL2SubgrpLem},  
$\pi_{\mathcal{P}} (g) = I_2$. 
This contradicts Lemma \ref{InjRadLem} 
for $n$ sufficiently large (in terms of $S$). 

Let $B$ be the set of all $K$-primes 
for which either (i) or (ii) fails, 
so that $\lvert B \rvert$ is bounded in 
terms of $S$ alone. 
Suppose for a contradiction that for all 
$p \in Q_n \setminus B$ 
and all $K$-primes 
$\mathcal{P} \in \mathcal{O}_K$ dividing $p$, 
there exists $1 \leq i \leq k$ such that 
$\pi_{\mathcal{P}} (c_i) \in Z(\SL_d (p))$. 
Thus for all $s \in S$, 
\begin{center}
$\pi_{\mathcal{P}} (s)\pi_{\mathcal{P}} (c_i)-\pi_{\mathcal{P}} (c_i)\pi_{\mathcal{P}} (s) = 0$ 
in $\mathbb{M}_d (p)$. 
\end{center}
By contrast, since all the $c_i$ are non-central 
in $\Gamma$ and $S$ generates $\Gamma$, 
we have for all $1 \leq i \leq k$ 
that there exists $s \in S$ such that 
$0 \neq sc_i - c_i s \in \mathbb{M}_d (K)$. 
Multiplying through by $\delta(c_i)\delta(s)$, 
we have $0 \neq s^{\ast}c_i ^{\ast}-c_i ^{\ast}s^{\ast} 
\in\mathbb{M}_d (\mathcal{O}_K)$, 
but all entries of $s^{\ast}c_i ^{\ast}-c_i ^{\ast}s^{\ast}$ are divisible by $\mathcal{P}$. 

Thus, taking the product of all nonzero 
entries of all matrices 
$s^{\ast}c_i ^{\ast}-c_i ^{\ast}s^{\ast}$ 
we obtain an element $z \in \mathcal{O}_K$ 
which is divisible by all of the $\mathcal{P}$. 
Thus, as observed above 
$p|N_{K/\mathbb{Q}}(\mathcal{P})|N_{K/\mathbb{Q}}(z)$ 
for each $p \in Q_n \setminus B$, so that: 
\begin{equation} \label{CentralLowerBd}
\frac{1}{C_0 ^{\lvert B \rvert}} 
\exp \big( cn^2 - 2 \lvert B \rvert \log(n) \big)
=\exp(cn^2) / (C_0 n^2)^{\lvert B \rvert} \leq \prod_{Q_n \setminus B} p 
\leq N_{K/\mathbb{Q}}(z)
\leq \mu(z)^D
\end{equation}
by Lemma \ref{ManySplitLem}. 
By contrast, by Lemma \ref{HeightBoundLem}, 
$\mu$ takes value at most $2 (dM)^{n+1}$ 
on the entries of 
$s^{\ast}c_i ^{\ast}-c_i ^{\ast}s^{\ast}$. 
Since $k \leq n$, 
there are at most 
$d^2 n$ nonzero elements of $\mathcal{O}_K$ 
which are entries of one of the 
$s^{\ast}c_i ^{\ast}-c_i ^{\ast}s^{\ast}$, 
and as such: 
\begin{equation}
\mu(z) \leq 2^{d^2 n} (dM)^{d^2 n(n+1)} 
\leq \exp\big( d^2(2 \log(dM) + \log(2)/n)n^2 \big)
\end{equation}
Taking $c$ sufficiently large in 
Lemma \ref{ManySplitLem} 
(which in turn requires 
$C_0$ to be chosen sufficiently large), 
this contradicts (\ref{CentralLowerBd}) 
for $n$ sufficiently large (in terms of $S$). 
\end{proof}

If $G$ is a finite group, 
recall that the \emph{diameter} of $G$ 
with respect to $T$ is: 
\begin{equation*}
\diam(G,T)=\min \lbrace r\geq 0 : B_T(r)=G \rbrace. 
\end{equation*}

\begin{propn} \label{LogDiamProp}
There exists $C=C(S)>0$ such that 
for all but finitely many rational primes $p$ 
which completely split in $K$, 
for all primes $\mathcal{P}$ of $\mathcal{O}_K$ 
which divide $p$, 
$\pi_{\mathcal{P}} (S)$ generates 
$\SL_d (p)$ and 
$\diam (\SL_d (p),\pi_{\mathcal{P}}(S))\leq C\log (p)$. 
\end{propn}

The proof rests on the following deep result, 
due to Helfgott \cite{Helf}, 
Breuillard-Green-Tao \cite{BGT} 
and Pyber-Szab\'{o} \cite{PySz}. 

\begin{thm} \label{ProductThm}
There exists $\epsilon = \epsilon(d)>0$ such that 
for any prime $p$ and any generating set 
$A \subseteq \SL_d (p)$, 
either: (a) $AAA = \PSL_d (p)$ or 
(b) $\lvert AAA\rvert\geq\lvert A\rvert^{1+\epsilon}$. 
\end{thm}

Here $AAA = \lbrace abc : a,b,c \in A \rbrace$. 
We also require a Lemma, 
the proof of which is similar to that 
of Proposition \ref{LogDiamProp}. 

\begin{proof}[Proof of Proposition \ref{LogDiamProp}]
First, as in Proposition \ref{GoodPrimesPropn}, 
we may assume that 
$\mathcal{P}$ does not divide 
any element of $\lbrace \delta(s) : s\in S \rbrace$ 
and that $\pi_{\mathcal{P}}(S)$ generates 
$\SL_2 (p)$. 
Next, $\Gamma$ has exponential word growth, 
that is, there exists $\delta=\delta(S)>0$ 
such that for all $r > 0$, 
$\lvert B_S (r) \rvert \geq (1+\delta)^r$ 
(this follows, for instance, from the Tits 
alternative, since $\Gamma$ has a nonabelian 
free subgroup). 
Let $\tilde{c}$ be as in Lemma \ref{InjRadLem} 
and let: 
\begin{center}
$A_0 =\pi_{\mathcal{P}}(B_S(\tilde{c}\log(p)))=B_{\pi_{\mathcal{P}} (S)}(\tilde{c}\log(p))$, 
$A_{m+1} = A_m A_m A_m$, 
\end{center}
so that for all $m$, 
$A_m \subseteq B_{\pi_{\mathcal{P}} (S)}(3^m \tilde{c}\log(p))$. 
Then by Lemma \ref{InjRadLem}, 
$\lvert A_0 \rvert \geq p^c$ for 
some $c=c(S)>0$. 
If none of $A_0 , A_1 , \ldots , A_m$ satisfy 
conclusion (a) of Theorem \ref{ProductThm}, 
then $\lvert A_m \rvert \geq p^{c (1+\epsilon)^m}$. 
This is impossible for $m \geq \tilde{C} = (2\log(d)-\log(c))/\log(1+\epsilon)$, 
since $\lvert \SL_d (p) \rvert \leq p^{d^2}$. 
Thus for some $m \leq \tilde{C}$, 
$\SL_d (p) \subseteq A_{m+1} \subseteq B_{\pi_{\mathcal{P}} (S)}(3^{\tilde{C}}\tilde{c}\log(p))$. 
Thus we have the desired bound, 
with $C = 3^{\tilde{C}}\tilde{c}$. 
\end{proof}

\begin{proof}[Proof of Theorem \ref{MainNoFieldThm}]
Let $p \in Q_n$ and $\mathcal{P}$ 
be as in Proposition \ref{GoodPrimesPropn}. 
Let $w(x) \in \Gamma \ast \mathbb{Z}$ 
be as in (\ref{worddefneqn}). 
For $1 \leq i \leq k$ let 
$\overline{c}_i \in\PSL_d(p)$ be the image 
in $\PSL_d(p)$ of $\pi_{\mathcal{P}}(c_i)\in\SL_d(p)$. 
By Proposition \ref{GoodPrimesPropn}, 
all $\overline{c}_i$ are nontrivial, 
so that: 
\begin{center}
$\overline{w}(x) 
= x^{a_0} \overline{c}_1 x^{a_1} 
\cdots \overline{c}_k x^{a_k} 
\in \PSL_d(p) \ast \mathbb{Z}$
\end{center}
is nontrivial. 
Choosing $C_0$ sufficiently large 
in the definition of $Q_n$, as we may, 
Theorem \ref{PSLThmBST} applies to $\overline{w}(x)$, 
so there exists $\overline{g} \in \PSL_d(p)$ 
such that $\overline{w}(\overline{g})\neq e$. 
Let $g \in \SL_d (p)$ be a lift of $\overline{g}$. 
Then noting that 
$\pi_{\mathcal{P}}(B_S(r)) 
= B_{\pi_{\mathcal{P}}(S)}(r)$, 
and applying Proposition \ref{LogDiamProp}, 
there exists $\tilde{g} \in B_S (C\log(p))$ 
such that $\pi_{\mathcal{P}}(\tilde{g}) = g$. 
Since $\pi_{\mathcal{P}}$ is surjective 
and $\pi_{\mathcal{P}} (w(\tilde{g}))$ 
is not central in $\SL_d(p)$, 
$w(\tilde{g})$ is not central in $\Gamma$. 
Finally, since $p \in Q_n$, 
$\log(p) = O(\log (n))$, as desired. 
\end{proof}

\begin{proof}[Proof of Theorem \ref{IntroMainThm}]
Let $\Gamma$ be as in the statement of 
the Theorem, and let $S$ be a finite generating 
set for $\Gamma$. 
Let $\iota : \SL_d(\mathbb{C}) 
\rightarrow \PGL_d(\mathbb{C})$ 
be the standard isogeny. 
Let $\hat{\Gamma} = \iota^{-1} (\Gamma)$ and 
$\hat{S} = \iota^{-1} (S)$. 
We note first that $\hat{\Gamma}$ 
is Zariski-dense in $\SL_d(\mathbb{C})$. 
For let $\mathbb{G} \leq \SL_d(\mathbb{C})$ 
be the Zariski 
closure of $\hat{\Gamma}$ in $\SL_d(\mathbb{C})$. 
Since $\iota$ is a morphism of algebraic groups, 
$\iota (\mathbb{G})$ is a closed subgroup 
of $\PGL_d(\mathbb{C})$ containing $\Gamma$, 
hence equals $\PGL_d(\mathbb{C})$. 
Thus $\dim(\mathbb{G}) \geq \dim (\iota (\mathbb{G})) 
= d^2 - 1$, and since 
$\SL_d(\mathbb{C})$ is a connected group 
of dimension $d^2 - 1$, we have 
$\mathbb{G} = \SL_d(\mathbb{C})$, as desired. 

Second, the adjoint representation of 
$\SL_d (\mathbb{C})$ factors through $\iota$. 
Thus, by condition (\ref{TraceEq}), 
$\hat{\Gamma}$ satisfies the hypotheses of 
Theorem \ref{MainNoFieldThm}. 
Let: 
\begin{center}
$w(x) = x^{a_0} c_1 x^{a_1} \cdots c_k x^{a_k} 
\in \Gamma \ast \mathbb{Z}$ 
\end{center}
with $\lvert w \rvert_{S \cup \lbrace x \rbrace}=n$ 
and $a_i \neq 0$ for $1 \leq i \leq k-1$. 
We can lift $w(x)$ to 
$\hat{w}(x) \in \hat{\Gamma} \ast \mathbb{Z}$ 
by replacing each $c_i$ by 
$\hat{c}_i \in \hat{\Gamma}$ satisfying 
$\iota(\hat{c}_i) = c_i$ and 
$\lvert \hat{c}_i\rvert_{\hat{S}}=\lvert c_i\rvert_S$. 
Then all $\hat{c}_i$ lie in 
$\hat{\Gamma} \setminus Z(\hat{\Gamma})$ and 
$\lvert\hat{w}(x)\rvert_{\hat{S}\cup\lbrace x\rbrace}
=n$. Applying Theorem \ref{MainNoFieldThm} 
(with $\hat{\Gamma}$; $\hat{S}$ and $\hat{w}(x)$ 
playing the r\^{o}les of $\Gamma$; $S$ and $w(x)$, 
we obtain $\hat{g} \in \hat{\Gamma}$ 
satisfying $\lvert \hat{g} \rvert_{\hat{S}} \leq C\log(n)$ and $\hat{w}(\hat{g})\notin Z(\hat{\Gamma})$. 
Then $g = \iota(\hat{g}) \in \Gamma$ satisfies 
$w(\iota(\hat{g}))= \iota (\hat{w}(\hat{g})) \neq e$ 
and $\lvert g \rvert_S \leq \lvert \hat{g} \rvert_{\hat{S}}$, 
so $g$ witnesses that $\chi_{\Gamma} ^S (w) 
\leq C\log(n)$. 
Since this holds for all such $w$, 
we conclude $\mathcal{M}_{\Gamma} ^S (n)\leq C\log(n)$, as desired. 
\end{proof}

\section{Applications} \label{AppSect}

As a first application of Theorem \ref{IntroMainThm}, 
we return to Example \ref{IntroPLSdEx}, 
and prove that $\PSL_d (\mathbb{Z})$ 
is sharply MIF. 
We shall be applying Theorem \ref{IntroMainThm} 
with $K = \mathbb{Q}$, 
so that hypothesis (\ref{TraceEq}) holds 
automatically.  
Finite generation of $\PSL_d (\mathbb{Z})$ 
is classical. 
Finally, since $\PSL_d (\mathbb{Z})$ is a 
lattice in $\PSL_d (\mathbb{R})$, 
it is Zariski-dense by the Borel Density Theorem. 

\begin{rmrk} \label{PSLdZFIRmrk}
Of course, the preceding argument 
extends to any finitely generated 
subgroup of $\PSL_d (\mathbb{Z})$ 
which is Zariski-dense in $\PSL_d (\mathbb{R})$. 
In particular, any finite-index subgroup of 
$\PSL_d (\mathbb{Z})$ is sharply MIF. 
\end{rmrk}

In the remainder of this Section we deduce 
Theorem \ref{IntroKleinMainThm} 
and Corollaries \ref{IntroFreeGroupsThm} 
and \ref{IntroMfldsCoroll} 
from Theorem \ref{IntroMainThm}. 

\subsection{Kleinian groups}

We refer to \cite{Series} for general background 
on Kleinian group. 
Our exposition in this Subsection is closely 
inspired by the proof of Theorem 1.2 
from \cite{LoLuRe}. 

\begin{thm} \label{KleinMainThm}
Let $\Gamma$ be a finitely generated 
torsion-free nonelementary Kleinian group. 
Then $\Gamma$ is sharply MIF. 
\end{thm}

As noted in \cite{Series}, 
examples of torsion-free nonelementary Kleinian group 
include all fundamental groups of 
closed orientable hyperbolic 2- 
and 3-manifolds. 
As such, Corollary \ref{IntroMfldsCoroll} 
follows immediately from Theorem \ref{KleinMainThm}. 

In proving Theorem \ref{KleinMainThm} 
is helpful to recall the following 
well-known characterization of Zariski-density 
in the case of $\SL_2$. 

\begin{thm}
For $\Gamma \leq \SL_2 (\mathbb{C})$, 
the following are equivalent. 
\begin{itemize}
\item[(i)] $\Gamma$ is Zariski-dense in $\SL_2 (\mathbb{C})$; 

\item[(ii)] $\Gamma$ has a nonabelian free subgroup; 

\item[(iii)] $\Gamma$ is not virtually soluble. 

\end{itemize}
\end{thm}

\begin{proof}[Proof of Theorem \ref{KleinMainThm}]
Suppose first that $\Gamma$ has finite covolume. 
Let $\tilde{\Gamma}$ be the preimage of $\Gamma$ 
in $\SL_2 (\mathbb{C})$. 
By Mostow-Prasad rigidity, 
up to conjugacy, we can assume 
that $\tilde{\Gamma} \leq \SL_2 (K)$ 
for some number field $K$. 
Moreover $\tilde{\Gamma}$ 
is not virtually soluble, 
so is Zariski-dense in $\SL_2 (\mathbb{C})$. 
Thus Theorem \ref{MainNoFieldThm} 
applies to $\tilde{\Gamma}$, 
and the conclusion for $\Gamma$ 
follows immediately. 

If $\Gamma$ has infinite covolume but is 
geometrically finite, 
then there exists a finite-covolume 
Kleinian group $\Gamma'$ such that 
$\Gamma$ is isomorphic to a subgroup of $\Gamma'$, 
by \cite{Brooks}. 
Then the faithful representation of $\tilde{\Gamma}'$ 
into $\SL_2(K)$ 
described in the previous paragraph 
restricts to $\tilde{\Gamma}$. 
Once again $\tilde{\Gamma}$ is not virtually soluble, 
so is Zariski dense in $\SL_2 (K)$, 
and Theorem \ref{MainNoFieldThm} applies. 

Finally if $\Gamma$ is geometrically infinite 
then it is isomorphic to a 
geometrically finite Kleinian group 
(see \cite{Ander}). 
\end{proof}

\subsection{Free groups}

\begin{thm} \label{FreeGroupsThm}
Every free group of finite rank $r \geq 2$ 
is sharply MIF. 
\end{thm}

By Theorem \ref{KleinMainThm}, 
it suffices to note that every finite-rank 
nonabelian free group is torsion-free 
nonelementary Kleinian. 
Alternatively, we can deduce 
Theorem \ref{FreeGroupsThm} 
using Remark \ref{PSLdZFIRmrk} and 
the following well-known result of Sanov 
\cite{Sanov}. 

\begin{thm} \label{SanovThm}
Let: 
\begin{center}
$a = \left( \begin{array}{cc}
1 & 2 \\
0 & 1 \\
\end{array} \right), 
b = \left( \begin{array}{cc}
1 & 0 \\
2 & 1 \\
\end{array} \right) \in \SL_2 (\mathbb{Z})$. 
\end{center}
Then $\langle a,b\rangle$ 
is a free group of rank $2$, 
of finite index in $\SL_2 (\mathbb{Z})$. 
\end{thm}

\begin{proof}[Proof of Theorem \ref{FreeGroupsThm}]
By Theorem \ref{SanovThm}, and the fact that 
the rank-$r$ free group $F_r$ 
embeds as a finite-index subgroup 
of $F_2$ for all $r \geq 2$, 
we can assume that $F_r$ is faithfully represented 
as a finite-index 
subgroup of $\SL_2 (\mathbb{Z})$. 
The restriction of the natural quotient 
$\SL_2 (\mathbb{Z}) \twoheadrightarrow 
\PSL_2 (\mathbb{Z})$ to $F_r$ is injective, 
since $F_r$ is torsion-free and $Z(\SL_2 (\mathbb{Z}))$ 
is finite, hence $F_r$ is also faithfully 
represented as a finite-index 
subgroup of 
$\PSL_2 (\mathbb{Z})$. 
The result follows as in Remark \ref{PSLdZFIRmrk}. 
\end{proof}

\section{Topological full groups}

In this section we prove Theorem \ref{NonSharpMIFThm}, 
which asserts that not every finitely generated 
MIF group is sharply MIF. 
Let $X$ denote the Cantor space. 

\begin{defn} \label{TFGdefn}
Let $G$ be a group and let 
$\alpha : G \rightarrow \Homeo (X)$ be a continuous 
action on $X$. 
The \emph{topological full group} $T(\alpha)$ 
of the action is the group consisting of all 
homeomorphisms $\phi$ of $X$ for which 
there exists a finite partition $X = C_1 \sqcup \cdots \sqcup C_p$ of $X$ consisting of nonempty clopen sets 
$C_i$, and elements $g_1 , \ldots g_p \in G$ 
such that $\phi|_{C_i} = \alpha(g_i)|_{C_i}$ 
for $1 \leq i \leq p$. 
\end{defn}

Recall that an action of a group $G$ on a set $\Omega$ 
is \emph{highly transitive} if, 
for every positive integer $n$, 
and every two $n$-tuples $x_1 , \ldots , x_n$ 
and $y_1 , \ldots , y_n$ of distinct points of $\Omega$, 
there exists $g \in G$ such that $g(x_i) = y_i$ 
for $1 \leq i \leq n$. 
It is easy to see that if the action of $G$ 
on $\Omega$ is highly transitive, 
then so is the restriction of that action to 
any finite-index subgroup of $G$. 

\begin{propn} \label{HTProp}
Let $G$ and $\alpha$ be as in Definition \ref{TFGdefn}. 
Let $x \in X$ and suppose the orbit $\Orb(x)$ of $x$ in $X$ 
is infinite. Then the action of $T(\alpha)$ 
on $\Orb(x)$ is highly transitive. 
\end{propn}

\begin{proof}
It suffices to prove that for any nonempty finite subset 
$\Sigma \subseteq \Orb(x)$, there is a subgroup 
$H(\Sigma)$ of $T(\alpha)$ preserving $\Sigma$ 
and acting on $\Sigma$ as $\Sym(\Sigma)$. 
Let $y,z \in \Sigma$ be distinct points, 
and let $g \in G$ with $\alpha(g)(y)=z$. 
There exists a clopen neighbourhood $C$ of $y$ such that 
$C \cap \alpha(g)(C) = \emptyset$; 
$C \cap \Sigma = \lbrace y \rbrace$ 
and $\alpha(g)(C) \cap \Sigma = \lbrace z \rbrace$. 
Let $\tau_{y,z} \in T(\alpha)$ fix 
$X \setminus (C \cup \alpha(g)(C))$,  
act on $C$ as $\alpha(g)$, and act on $\alpha(g)(C)$ 
as $\alpha(g)^{-1}$. Then $\tau_{y,z}$ preserves $\Sigma$ 
and acts upon it as the transposition $(y \; z)$. 
The desired result follows. 
\end{proof}

An action of a group on $X$ is \emph{minimal} 
if every orbit of the action is dense in $X$. 

\begin{thm}[Theorem 8.1 of \cite{Nekr}] \label{NekraThm}
There exists a finitely generated 
infinite group $F$ of subexponential word growth 
with a faithful continuous minimal action 
$\alpha : F \rightarrow \Homeo (X)$, 
such that $T(\alpha) \cong F$. 
Moreover the derived subgroup $[F,F]$ of $F$ is 
simple, and has finite index in $F$. 
\end{thm}

\begin{proof}[Proof of Theorem \ref{NonSharpMIFThm}]
Let $F$ be as in Theorem \ref{NekraThm}. 
Our example shall be $\Gamma = [F,F]$. 
Being of finite index in $F$, 
$\Gamma$ is also a finitely generated group of 
subexponential growth. 
As noted in \cite{Brad} Remark 9.3, 
no group of subexponential growth can be sharply MIF. 
It therefore suffices to verify that $\Gamma$ is MIF. 

Let $x \in X$ and set $\Omega = \Orb(x)$, 
a countably infinite set. 
Then by Proposition \ref{HTProp} and $F \cong T(\alpha)$, 
$F$ admits a highly transitive action on $\Omega$, 
hence so does $\Gamma$. Moreover, 
since $\Omega$ is dense in $X$, 
under this action no nontrivial 
element of $\Gamma$ has finite support on $\Omega$. 
Finally, we apply Theorem 5.9 of \cite{HullOsin}: 
if $\Gamma$ is not MIF, 
then $\Gamma$ contains a normal subgroup isomorphic to 
the group $A$ of all finitely supported even permutations 
of $\mathbb{N}$. Since $\Gamma$ is simple, 
we have $\Gamma \cong A$. 
This is a contradiction, as $\Gamma$ is finitely generated 
and $A$ is not. 
\end{proof}

\begin{rmrk}
\normalfont
As noted in the Introduction to \cite{Nekr}, 
the construction given therein yields an upper 
bound on the word-growth of $\Gamma$ of the form 
$f(n) = C_1 \exp\big(n/\exp(C_2 \sqrt{\log n})\big)$. 
By the argument of \cite{Brad} Remark 9.3, 
we obtain a lower bound for $\mathcal{M}_{\Gamma}$ 
which is approximately an inverse function to $n^2 f(n)$. 
The lower bound thus obtained is stronger than
 $\log (n) \log\log (n)^C$, 
but weaker than $\log(n)^{1+\frac{1}{C}}$, 
for every $C>0$. 
Presumably this lower bound is far from sharp; 
MIF growth for topological full groups 
should be investigated elsewhere. 
\end{rmrk}

\subsection*{Acknowledgements}

We are grateful to Nicol\'{a}s Matte Bon for 
drawing our attention to 
the relevance of Nekrashevych's construction 
to proving Theorem \ref{NonSharpMIFThm}.

\end{document}